\documentclass[SEDP]{cedram-sem}

\usepackage[latin1]{inputenc}

 \usepackage[T1]{fontenc}

 \usepackage{lmodern}

\usepackage{amsfonts,amsmath,amssymb,mathrsfs,mathbbol,bbm,mathtools}
\usepackage{amsthm}
\theoremstyle{plain}
\newtheorem{theorem}{Theorem}
\newtheorem{proposition}[theorem]{Proposition}

\theoremstyle{definition}

\newtheorem{definition}[theorem]{Definition}

\newcommand{\deriv}[1]{\frac{\mathrm{d}}{\mathrm{d}#1}}
\newcommand{\derivk}[2]{\frac{\mathrm{d}^{#2}}{\mathrm{d}#1}}
\newcommand{\initState}				{\state^0}

\newcommand{\state}					{y}
\newcommand{\criter}				{\mathscr{J}}
\DeclareMathOperator*{\argmin}		{argmin}
\newcommand{\abs}[1]				{\left| #1 \right|}
\newcommand{\dst}{\displaystyle}

\newcommand{\f}{\frac}
\newcommand{\ep}{\varepsilon}
\newcommand{\vareps}{\varepsilon}
\newcommand{\1}{{\mathchoice {\rm 1\mskip-4mu l} {\rm 1\mskip-4mu l}{\rm 1\mskip-4.5mu l} {\rm 1\mskip-5mu l}}}
\def\supp{{\mbox{\small\rm  supp}}\,}
\def\M{\mathcal{M}}

          \newcommand{\solFinf}				{u_{\infty}}

\newcommand{\R}						{\mathbbm{R}}
          \newcommand{\sol}					{u}
\newcommand{\solSinf}				{u^{\varepsilon}_{\infty}}
\newcommand{\norm}[1]				{\left\|#1\right\|}  
\newcommand{\Hone}[1][0,L]{\mathrm{H}^1(#1)}
\newcommand{\HoneR}[1][0,L]{\mathrm{H}^1_R(#1)}
\newcommand{\Ltwo}[1][0,L]{\mathrm{L}^2(#1)}
\newcommand{\Czero}[1][0,T]{\mathrm{C}^0(#1)}

\newcommand{\uin}{\sol^0}
\newcommand{\Htwo}[1][0,L]{\mathrm{H}^2(#1)}
\newcommand{\Hs}[1][0,L]{\mathrm{H}^s(#1)}
\newcommand{\Cone}[1][0,T]{\mathrm{C}^1(#1)}
\newcommand{\N}						{\mathbbm{N}}
\newcommand{\Cst}					{C^{\text{st}}}
\newcommand{\p}{\partial}
          \newcommand{\OpDir}				    {\Psi}
\newcommand{\OpD}				    {\OpDir}
\newcommand{\solS}					{u}
\newcommand{\solF}					{u}
\newcommand{\diff}					{\mathrm{d}}     
\newcommand{\observ}				{y}

\title[Inverse problems for depolymerisation and fragmentation]
{Moments approaches for asymptotic inverse problems of depolymerisation and fragmentation systems
}

\alttitle{In collaboration with Miguel Escobedo, Philippe Moireau and Magali Tournus}

\author
{\firstname{Marie} \lastname{Doumic}}

\address{MERGE, CMAP, Inria, \\ IP Paris, Ecole polytechnique, CNRS, \\
91128 Palaiseau cedex\\
FRANCE}
\thanks{This work has been partially supported by  the project ODISSE ANR-19-CE48-0004}

\email{marie.doumic@inria.fr}
  

\begin{document}


\begin{abstract}  Shrinkage of large particles, either through depolymerisation (i.e. progressive shortening) or through fragmentation (breakage into smaller pieces) may be modelled by discrete equations, of Becker-D\"oring type, or by continuous ones. In this note, we review two kinds of inverse problems: the first is the estimation of the initial size-distribution from moments measurements in a depolymerising system, in collaboration with Philippe Moireau and inspired by experiments carried out by Human Rezaei's team; the second is the inference of fragmentation characteristics from size distribution samples, in collaboration with Miguel Escobedo and Magali Tournus, based on biological questions and experiments of Wei-Feng Xue's team.
 \end{abstract}

\maketitle 
 
\section{Introduction}
Polymers are large macromolecules formed out of small molecular units, called monomers. They are ubiquitous in nature and industry - e.g. plastics, biopolymers such as DNA, actin filaments or protein fibrils. Their shortening,  either through depolymerisation (i.e., loss of monomers) or through fragmentation 
into smaller pieces are dynamical phenomena which appear in many applications. More specifically,
the departure point of our research has been protein fibrils depolymerisation and breakage, thought to be key
mechanisms for many diseases (Parkinson's, Alzheimer's, Creuzfeldt-Jakob's etc.) as well as for many functional biomolecular
systems (actin filaments). The dynamic nature of the experiments, as well as their nanoscale, makes it
very challenging to estimate their features, i.e. their reaction rates or size distributions, leading to an urgent need for mathematical models and estimation methods to be developed. 
In this note, we consider two case studies of such inverse problems, directly inspired by modern experimental setups. 

\

The first problem is based on experiments  carried out in H. Rezaei's lab (INRAE, Jouy-en-Josas, France) where the observation consists in the time evolution of a moment of the size distribution  of protein polymers called PrP, responsible for Prion diseases. The moment observed may be the total polymerised mass (first moment) or the average molecular weight (second moment). The original model is a discrete depolymerisation system, based on the constant coefficient case of the Becker-D\"oring equations, namely the reactions

\begin{equation}\label{eq:depol}
{\mathcal C}_i \longrightarrow {\mathcal C}_{i-1} + {\mathcal C}_1.\end{equation}

\

 Our aim is to estimate the initial polymer size distribution from such moment observations. 
We first evaluate the impact of using continuous approximations of the initial discrete model to solve this
inverse problem. At first order, the model is approximated by a backward transport equation, for which
the inverse problem turns to be mildly ill-posed (of order $k+1$ when used to invert the time evolution
of the $k$th-moment of the solution). This remains true when polymerisation is also considered~\cite{DellaValleDoumicMoireau2024}, as in the
full Becker-D\"oring system, though the inversion reveals more intricate due to the fact that the problem
becomes nonlinear.

At second order, the asymptotic model becomes an advection-diffusion equation, where the diffusion
is a corrective term, complemented with an original transparent boundary condition at $x = 0$. This
approximation is more accurate, but we face an accuracy versus stability trade-off: the
inverse reconstruction reveals to be severely ill-posed. Thanks to Carleman inequalities and to log-convexity
estimates, we prove observability results and error estimates for a Tikhonov regularisation.
We then develop a Kalman-based observer approach, which reveals very efficient for the numerical
solution. This is a joint work with Philippe Moireau (Inria)~\cite{DoumicMoireau2024}, inspired by depolymerisation experiments carried out
by Human Rezaei and collaborators (Inrae). We sketch the main results of this study in Section~\ref{sec:depol}.

\

The second problem is based on fragmentation experiments carried out on several protein polymers  by W.-F. Xue's team (Univ. of Kent, Canterbury, United Kingdom), namely reactions of the kind

\begin{equation}\label{eq:frag}
{\mathcal C}_i \longrightarrow {\mathcal C}_{i-j} + {\mathcal C}_j.\end{equation} 

\

The biophysical question which interested our collaborators was to estimate the (size-dependent) fragmentation {\it rate}, as well as the so-called fragmentation {\it kernel}, which characterise the stability of the polymers and the places where they are more likely to break. We have proposed several approaches based on the continuous fragmentation equation, studying and making use either of the long-term, the transient or the short-term dynamics. Error estimates in Bounded Lipshitz norm are obtained for this last approach. This is a joint work with Miguel Escobedo and Magali Tournus~\cite{doumic2021inverse}, and the project is a long-standing collaboration with Wei-Feng Xue and collaborators~\cite{doumic:hal-01501811,tournus2021insights,BTMPSTDX19}, that we develop in Section~\ref{sec:frag}.


%
%

\section{Asymptotic inverse problems for depolymerisation systems}
\label{sec:depol}

\subsection{Original inverse problem: a discrete system}

At the basis of this research program lies experimental protocols which follow the time dynamics of average quantities over a size distribution of polymers. Typical measurements consist either in the total polymerised mass (e.g. through a Thioflavine T or Th.T protocol~\cite{prigent:hal-00778052}) or the average molecular weight (e.g. with Static Light Scattering~\cite{armiento:hal-01574346}). In a discrete setting, denoting $C_i(\tau)$ the concentration of polymers containing $i$ monomeric units, such measurement may be modelled by
\begin{equation}\label{def:moments}M_k (\tau)\coloneqq \sum\limits_{i=i_0}^\infty i^k C_i(\tau),
\end{equation}
with $k=1$ for the polymerised mass, $k=2$ for the average weight, $i_0\geq 1$ either a detection level or the smallest stable polymer, and the measurement of $M_k$ is made up to a noise. The first questions asked by our biologist collaborators were: what can we identify from such measurements, and what cannot we? We have addressed this question in different experimental settings~\cite{armiento:hal-01574346,Kruse2}, and then decided to focus on the question of initial-state observability and estimation, 
both because it is  interesting in itself and because it is a prerequisite to parameter estimation: even in cases where a priori knowledge of the initial state is known, it is always partial and noisy~\cite{armiento:hal-01574346}, requiring stability analysis.

We assume a constant depolymerisation rate $b>0,$  so that the mass balance equation of~\eqref{eq:depol} leads to the elementary system
\begin{equation}\deriv{\tau} C_i=b(C_{i+1}-C_i),\qquad C_i(0)=C_i^0.
\end{equation}
\subsection{Asymptotic approximations by continuous equations}
A fundamental characteristics of the systems we are interested in is that the average polymer size is very large - from some hundreds to some thousands of hundreds of monomers. For this reason, we rescale the system by accelerating the time, which is an alternative of size rescaling as done in~\cite{Thierry} for instance. We define, for a given $\ep \ll 1,$ $t\coloneqq \ep \tau \in [0,T],$ $c_i^\ep (t) \coloneqq  C_i(\f{t}{\ep}),$ and obtain the rescaled system
\begin{equation}\label{eq:discrete}
\deriv{t}c_i^\ep=\f{b}{\ep}(c_{i+1}^\ep-c_i^\ep),\qquad c_i^\ep(0)=c_i^0,\qquad i\geq i_0.
\end{equation}
We recognise a first-order finite difference scheme for the transport equation, which led us to define a stepwise interpolant and  grid:
\begin{equation}\label{def:grid}\forall\;i\geq i_0,\quad x_i^\ep\coloneqq \ep (i-i_0),\qquad u^\ep(t,x)\coloneqq c_i^\ep,\qquad t\geq 0,\; x\in [x_i^\ep,x_{i+1}^\ep).\end{equation}
A Taylor expansion leads us to the following backward transport equation as a first order approximation:
\begin{equation}
\label{eq:first-order}
\left\{
\begin{aligned}
	\partial_t \solFinf - b \partial_x \solFinf&=0,&  (t,x) \in  [0,T]\times \R^+ ,\\
	\solFinf(0,x) &= \sol^0(x), & x \in \R^+,
\end{aligned}\right.
\end{equation}
whereas the second order approximation is given by
\begin{equation}\label{eq:second-order}
\left\{
\begin{aligned}
        \partial_t \solSinf-b \partial_x\solSinf- \f{b\ep}{2} \partial_{x}^2 \solSinf &= 0, & (t,x)\in  [0,T]\times\R^+, \\
        \partial_t\solSinf(0,t) - b \partial_x \solSinf(0,t) &=0, & t\in[0,T],  \\
        \solSinf(0,x) &= \sol^\vareps(0,x),&  x\in \R^+.
\end{aligned}
\right.
\end{equation}
The transport-diffusion equation comes from the Taylor expansion, whereas the transport boundary condition at $x=0$ may be obtained by two considerations:
\begin{itemize}
\item as a boundary layer, we keep only the first order term, see e.g.~\cite{halpern1986artificial},
\item we require that the balance for the total number of polymers is preserved. This is given in the discrete setting by
\[\deriv{t}{M_0^\ep}=\deriv{t}\sum\limits_{i=i_0}^\infty c_i^\ep =-bc_{i_0}^\ep (t), \]
and doing the asymptotic expansion on this equality, where we identify $c_{i_0}^\ep$ with $u_\infty^\ep(t,0),$ leads to the transport boundary condition.
\end{itemize}
Well-posedness for~\eqref{eq:second-order} follows for instance from the Lumer-Phillips theorem applied to a well-chosen maximal accretive operator. To state error estimates, we introduce the following discrete norm
\[
 \norm{u}_{2,\vareps}^2 \coloneqq \sum_{i \geq i_0} \vareps  |u(x_i^\vareps)|^2.
\]
This norm is well defined for functions in $\Hone[\R_+^*]$ and is consistent with the $\Ltwo[\R_+]$-norm as $\varepsilon$ tends to 0. Moreover, it is well defined and equal to the $\Ltwo[\R^+]-$norm for piecewise constant functions defined on the grid $(x_i^\ep),$ as $u^\ep(t,\cdot)$ defined by~\eqref{def:grid} is. 

We prove the following error estimates for the two approximate systems.

\begin{proposition}[Prop.~2.1. and~2.2. of~\cite{DoumicMoireau2024}]\label{prop:estim}
Let  $\uin \in \mathrm{H}^2(\R^+)$ and  $\solFinf\in  \Czero[(0,T);\mathrm{H}^2(\R_+^*)] \cap \Cone[(0,T);\Ltwo[\R_+]]$ solution to~\eqref{eq:first-order}, $\solSinf\in 	 \Czero[(0,T);\Htwo[\R^+]] \cap \Cone[(0,T);\Ltwo[\R^+]]$ solution to~\eqref{eq:second-order} with $\uin$ as an initial condition. Let $(c_i^{0,\ep})\in \ell_2 (\N)$, $(c_i^\ep)$ solution to~\eqref{eq:discrete} and $\sol^\ep$ defined by~\eqref{def:grid}  such that 
\[\|\sol^\vareps(0,\cdot)- \sol^0 \|_{2,\vareps} \leq \varepsilon^\alpha,\] 
for $\alpha>0.$ There exists a constant $\Cst>0$ depending only on $b$ such that for all $t > 0$,   we have
 \[
	 \norm{ \sol^\vareps(t,\cdot)- \solFinf(t,\cdot)}_{2,\vareps} \leq  {\ep^\alpha+ \Cst \Vert\sol^0\Vert_{\Htwo}\,\ep\, t},
	 \]
	 and
	 \[
	 \norm{ \sol^\vareps(t,\cdot) - \solSinf(t,\cdot)}_{2,\vareps} \leq \ep^\alpha + \Cst \Vert\sol^0\Vert_{\Htwo} \,\vareps^{\f{3}{2}}\,t^{\f{1}{2}} .
	 \]
\end{proposition}
The proof is obtained through the Taylor expansion, using the Cauchy-Schwarz inequality and a Gronwall lemma. For the second order approximation, we also use an energy estimate for $\p^2_{xx} \solSinf,$ which satisfies the same transport-diffusion equation as $\solSinf$ complemented with an homogeneous Dirichlet boundary condition at $x=0.$

\subsection{Original and asymptotic inverse problems}
The two approximations~\eqref{eq:first-order} and~\eqref{eq:second-order} may then be used to solve the original discrete inverse problem, which states: From a noisy measurement of~\eqref{def:moments} for $\tau \in [0,\mathcal{T}]$, how to estimate the initial distribution~$C_i(0)$? We first apply the rescaling by $t=\ep \tau$ and define
\begin{equation}\label{def:moments:eps}
M_k^\ep (\tau)\coloneqq \ep \sum\limits_{i=i_0}^\infty (\ep i)^k c_i^\ep (t),
\end{equation}
so that we have the following moments dynamics
\begin{equation}\label{eq:moments}	\deriv{t} M_k^\ep = 
	\begin{cases}
		-b c_{i_0}^\ep , & \text{if } k=0, \\
		-b M_0^\ep -b\ep i_0 c_{i_0}^\ep , & \text{if } k=1, \\
		- b k  M_{k-1}^\ep+O(\ep^2), & \text{if } k\geq 2.		
	\end{cases}
\end{equation}
The term $O(\ep^{2})$ is formally obtained, but for a compactly supported initial condition, Prop.~\ref{prop:estim} implies that it is at most $O(\ep^{3/2}).$ 
We now restrict ourselves to such cases, namely $\ep i\leq L,$ or equivalently $x\in [0,L].$ 

The original inverse problem may be formulated as: Invert 
\[
\OpD_{T,\ep}^{\uin\to  M_k^\ep}: \left|
	\begin{aligned}
		\Ltwo[0,L] &\to \Ltwo[0,T] \\
		\uin &\mapsto M_k^\ep. 
	\end{aligned}\right.
	\]
The system~\eqref{eq:moments} links $M_k^\ep$ to $c_{i_0}^\ep(t)=\solS(t,0),$ so that writing $\OpD_{T,\ep}^{\uin\to  M_k^\ep} = \OpD_{T,\ep}^{\uin\to \text{Tr}} \circ \OpD_{T,\ep}^{{Tr}\to M_k^\ep}$ we decompose our problem  into two steps:
\begin{enumerate}
\item \label{step:1} Invert 
\[\OpD_{T,\ep}^{\uin\to \text{Tr}}: \left|
	\begin{aligned}
		\Ltwo[0,L] &\to \Ltwo[0,T] \\
		\uin &\mapsto \solS(\cdot,0) \in \R)
	\end{aligned}\right.
	\]
	\item \label{step:2} Invert
	\[\OpD_{T,\ep}^{{Tr}\to M_k^\ep}: \left|
	\begin{aligned}
		\Ltwo[0,T] &\to \Ltwo[0,T] \\
		\solS(\cdot,0) &\mapsto M_k^\ep. 	\end{aligned}\right.
	\]
\end{enumerate}
In the case of the first order inverse problem, the method of characteristics implies that $\solS(t,0)=\uin(bt)$ so that inverting $\OpD_{T,\ep}^{\uin\to \text{Tr}}$ is a well-posed problem if and only if $b T\geq L.$ Moreover, 
we may neglect the term  $ -b\ep i_0 c_{i_0}^\ep$ in~\eqref{eq:moments} and  obtain
\[\derivk{t^{k+1}}{k+1}{M_k^\ep} (t) = (-b)^{k+1} k! c_{i_0}^\ep (t) +O(\ep)= (-b)^{k+1} k! \uin (bt) +O(\ep),
\]
which shows that inverting $\OpD_{T,\ep}^{{Tr}\to M_k^\ep}$ is {\it moderately/mildly} ill-posed of degree  $k+1,$ and from which we immediately infer the following observability/stability inequality
\begin{equation}\label{eq:stability-order-0}
	\forall T \geq T_0 \coloneqq \frac{L}b, \quad \| \solF^0 \|^2_{\Ltwo} \lesssim \norm{\frac{\diff^{k{+1}}}{\diff t^{k{+1}}} {M}_k}_{\Ltwo[0,T]}^2,
\end{equation}
see~\cite{armiento:2016} for more details and a full solution of the inverse problem - theoretical, numerical with the implementation of a Kalman-type sequential approach, and also applied to experimental data. Let us only mention that for $\uin \in \Hs$, $s>0,$ an observation $\observ^\delta$ and a measurement error 
\[\Vert M_k^\ep - \observ^\delta \Vert_{\Ltwo} \leq \delta, \]
 we obtain an optimal error estimate in the order of $\ep^{\f{s}{k+s+1}}$ - as expected given the degree of ill-posedness $k+1.$

In the case of the second order problem, the step~\eqref{step:1} reveals {\it severely} ill-posed, as linked to the time-reversal of  the (infinitely smoothing) heat equation. On the contrary, the second step, given by inverting~\eqref{eq:moments}, is slightly less ill-posed, of degree $k$ instead of $k+1$ if $k\geq 1,$ thanks to the corrective term $ -b\ep i_0 c_{i_0}^\ep$.  In the following, we thus focus on Step~\eqref{step:1}: inverting $\uin \mapsto \solS(\cdot,0)$.

\subsection{Observability of the second-order inverse problem}
A necessary step before solving the inverse problem is to prove that it is observable - other said, that we have a unique solution in a certain space and continuity in a certain sense. To do so, we first restrict ourselves to a bounded domain $[0,L]$, and look for solutions to
\begin{equation}\label{eq:second-order-L}
\left\{
\begin{aligned}
        \partial_t \solS-b \partial_x\solS- \f{b\ep}{2}\partial_{x}^2 \solS &= 0, & (t,x)\in [0,T]\times [0,L], \\
        \partial_t\solS(t,0) - b \partial_x \solS(t,0) &=0, & t\in[0,T],  \\
        \solS(t,L) &=0, & t\in[0,T], \\
        \solS(0,x) &= \uin(x),&  x\in (0,L).
\end{aligned}
\right.
\end{equation}
The homogeneous Dirichlet boundary condition at $x=L$ comes naturally from the discrete problem, for which,  if initially, for $i\geq \ep^{-1}L,$ we have $c_i^\ep (0)=0,$  then $c_i^\ep (t)=0$ for any $t\geq 0.$ In this respect, the bounded problem is thus even more faithful than the unbounded one for compactly-supported initial data. The system is solved, as for the unbounded problem~\eqref{eq:second-order}, by introducing an appropriate accretive operator, and the following observability inequality is obtained. 
\begin{theorem}[Th.~4.5 in~\cite{DoumicMoireau2024}]\label{th:initobserv}
Let $\uin \in \HoneR\cap \Htwo$, with $\Vert \uin \Vert_{\Hone} \leq M$. 

Let $\sol \in \Czero[(0,T);\HoneR]\cap \Cone[(0,T);\Ltwo]$ the unique solution to~\eqref{eq:second-order-L}. 

Defining the increasing function $\rho: x\mapsto x e^x,$ we have the following observability inequality, for constants $c, C$ depending only on $b$ and $L:$
 \begin{equation}\label{estim:observ:final}
 \left\|\sol(0,\cdot) \right\|_{L^2(0,L)}^{2}+{\ep}\vert \sol (0,0)\vert^2  \lesssim \f{Ce^{\f{L}{\ep}}M^2}{
 \rho^{-1} \left(C T \f{e^{-\f{c}{\ep} (1+T^{-1})}}{1+T^{-4}} \f{ M^2}{\int_0^T \vert \sol\vert^2 (t,0) \diff t}\right)
 } \f{T}{\ep}.
 \end{equation}
\end{theorem}
{\bf Sketch of the proof.} The inequality~\eqref{estim:observ:final} relies on two main ingredients: 
\begin{itemize}
\item a log-convexity inequality, inspired by~\cite{Bardos:1973aa,Phung:2004aa,vo:tel-02081052}, which gives a bound of the initial condition with respect to the final time solution. This allows us to bypass a more standard exponential stability estimate, which bounds the final time solution by the initial condition.
\item A Carleman inequality, inspired by~\cite{Cornilleau:2012dba,Coron:2005vd}, which proves a controllability inequality for a very similar system. This technical  part is an  adaptation of the proof of Proposition~10 of~\cite{Cornilleau:2012dba}. 
\end{itemize}
A first interest of the inequality~\eqref{estim:observ:final} is to provide uniqueness for the inversion of $\OpD_{T,\ep}^{\uin\to \text{Tr}}$ and stability of a reconstruction. Moreover,  thanks to the diffusion term, 
 we no longer have a lower-bound condition  on $T$ like the condition $T\geq T_0$ for~\eqref{eq:stability-order-0}. There is a price to pay however, which is the logarithmic rate of convergence given by $1/\rho^{-1}$. 
\subsection{Solution and error estimate for the second-order inverse problem}

A second benefit of~\eqref{estim:observ:final} is to indicate how a Tikhonov regularisation strategy may be built, and an estimate for the reconstruction. Since the bound $M$ of the $\Hone$ norm appears in~\eqref{estim:observ:final}, we define the following quadratic functional to minimise:
\begin{equation}\label{eq:least-squares functional}
	\criter_{|T} (\uin) := \frac{1}{2M^2} \Vert \uin\Vert_{H^1}^2
	+   \frac{1}{2\delta^2} \int_0^T  \abs{\observ^\delta(t) - \sol_{|\uin}(t,0)}^2  \,  \diff t ,		
	\end{equation}
 and we obtain the following error estimate.
\begin{theorem}[Th.~4.6 in~\cite{DoumicMoireau2024}]\label{th:tikhonov}
Let $\uin \in \HoneR$ with $\Vert \uin\Vert_{H^1}\leq  M,$ $\observ= \sol_{\uin}(t,\cdot)=\OpD_{T,\ep}^{\uin\to \text{Tr}}(\uin),$ $\delta>0$. We assume that we observe $\observ^\delta$ such that
\[\Vert \observ - \observ^\delta\Vert_{L^2(0,T)} \leq \delta.
\]
Let $\widehat{\uin}$ an estimate for $\uin$ defined by
\begin{equation}\label{def:uinestim}
\widehat{\uin}=\argmin_{{\uin}\in \HoneR}  \criter_{|T} (\uin).
\end{equation} 
Then there exist constants $C_1$ and $C_2,$ depending only on the parameters $L,$ $b$, $T$ and $\ep,$ such that
\begin{equation}
\Vert \uin - \bar{\sol}^0_{|T}\Vert_{L^2(0,L)}^2 \leq \f{C_1 M^2}{{\rho^{-1}(C_2 \f{M^2}{\delta^2})}}.
\end{equation}
\end{theorem}
As expected, the speed of convergence is logarithmic with respect to the noise $\delta,$ to be compared with an algebraic rate of convergence for the first-order system. We give here the proof, which is very short and emblematic of error estimates for Tikhonov regularisation method.
\begin{proof}
Let us denote $\tilde{\sol}^0=\uin-\widehat{\uin}$ and  $\tilde\sol=\sol_{|\tilde{\sol}^0}$ the solution of~\eqref{eq:second-order-L} with $\tilde{\sol}^0$ as an initial condition. We apply~\eqref{estim:observ:final} to $\tilde\sol$ (and replace $M$ by $\Vert\tilde{\sol}^0\Vert_{H^1}$), and obtain
\[\left\|\tilde\sol(\cdot,0) \right\|_{L^2(0,L)}^{2}+{\ep}\vert \tilde\sol (0,0)\vert^2  \lesssim \f{Ce^{\frac{L}{\ep}}\Vert \tilde{\sol}^0\Vert_{H^1}^2}{
 \rho^{-1} \left(C T \f{e^{-\f{c}{\ep} (1+T^{-1})}}{1+T^{-4}} \f{ \Vert\tilde\sol^0\Vert_{H^1}^2}{\int_0^T \vert \tilde\sol\vert^2 (t,0) \diff t}\right)
 } \f{T}{\ep}.
\]
From $\criter_{|T} (\widehat{\uin}) \leq \criter_{|T} (\uin) \leq 1$ we have
\[\int_0^T \vert \tilde\sol\vert^2 (t,0) \diff t \leq 2 \int_0^T \vert u_{|\widehat\uin} - \observ^\delta \vert^2 (t,0) \diff t + 2 \int_0^T \vert \observ^\delta -\observ \vert^2 (t,0) \diff t \leq 6 \delta^2.
\] 
By the triangular inequality, we obtain
\[\Vert\tilde{\sol}^0\Vert_{H^1}^2\leq 2 \Vert\widehat{\uin}_{|T}\Vert_{H^1}^2+2\Vert\uin\Vert_{H^1}^2\leq 2 M^2+ 2\Vert\widehat{\uin}_{|T}\Vert_{H^1}^2.
\]
To estimate $\Vert \widehat{\uin}_{|T}\Vert_{H^1}^2,$ we use the fact that it minimises $\criter_{|T}:$ by definition, we have
\begin{multline*}
\Vert\widehat{\uin}\Vert_{H^1}^2 
\leq {2M^2}\criter_{|T}(\bar\initState)\leq {2M^2}\criter_{|T}(\uin)
\\
\leq {{2M^2}\left(\f{1}{2M^2} \Vert \uin\Vert^2_{H^1}+\f{1}{2}\right)\criter_{|T}(\bar\initState)}
\leq 2 M^2 ,
\end{multline*}
hence
\[\Vert\tilde{\sol}^0_{|T}\Vert_{H^1}^2\leq 4 M^2
\]
and we conclude by the fact that $x\mapsto \f{x}{\rho^{-1}(x)}$ is increasing on $(0,\infty)$: we compute that $\rho'(x)=e^x(1+x)$, hence taking $x=\rho(y)$ we get
\[ \f{\diff}{\diff x}(\f{\cdot}{\rho^{-1}(\cdot)})\left(\rho(y)\right)=\f{1}{y}-\f{\rho(y)}{y^2 \rho'(y)}=\f{1}{y} - \f{1}{y(1+y)}=\f{1}{1+y}>0
\] for $y>0$.
\end{proof}
To conclude, we proposed a sequential approach  based on Kalman filtering to solve the estimation problem. Despite the severely ill-posed character of the second-order problem, it reveals more accurate in simulation tests than the pure transport approximation. Moreover, the information contained in the corrective diffusion is visible in the fact that even for $T\leq T_0=\f{L}{b}$ we are able to estimate the distribution for $x>bT.$ How to refine the Carleman estimate cited above, in order to quantify precisely how this information content vanishes for small times and, on the contrary, stop improving for $T_0\lesssim T,$ remains a very interesting open problem.

\section{Inverse problems for the fragmentation equation}\label{sec:frag}
Let us now turn to reactions of fragmentation type as given by~\eqref{eq:frag}. Writing a first-order asymptotics in the same spirit as~\eqref{eq:first-order} leads to the fragmentation equation:
\begin{equation}\label{eq:frag_intro}
\dfrac{\p}{\p t}u(t,x) =- \alpha x^\gamma u(t,x) +  \alpha \dst\int_{x}^{\infty}\kappa\left(\dfrac{x}{y} \right)y^{\gamma-1} u(t,y){\diff y},
 \end{equation}
where $\alpha x^\gamma$ represents the breakage rate of a particle of size $x$, that we have assumed to be given by a power law, and $\kappa(z)$ is linked to the probability for a particle of size $y$ to give rise to a particle of size $x=zy.$ This model appeared as the best-fit one in experiments on $\beta_2-$microglobulin~\cite{XueRadford2013}, where the experimental measurements consisted in samples of fibril sizes observed at several time points. This led us to formulating the following inverse problem: How to estimate the fragmentation features given by the parameters $(\alpha,\gamma,\kappa)$ from such data,   with $\alpha,\gamma >0$ and $\kappa$ a probability measure on $[0,1]$?

For mass conservation considerations we assume here binary fragmentation and no atom at $z=0$ and $z=1,$  hence
\begin{equation}\label{as:kappa}
\kappa \in \mathcal{M^+}((0,1)),  \qquad 
\int_0^1 \kappa(\diff z)=2,\qquad \int_0^1 z \kappa(\diff z)= 1,\qquad \kappa(z)=\kappa(1-z).
\end{equation}
For simplicity we have assumed binary fragmentation, which corresponds to $N=2$ in~\cite{doumic2021inverse}, but the study is unchanged. 

This section explains the various stages in our approach, and how we successively used asymptotic behaviour, then large-time dynamics, and finally short-time behaviour, to try to estimate $\kappa,\,\alpha$ and $\gamma$. The common thread running through our studies is the use of moments of the solution to interpret the data observed, either integer moments or through the Mellin transform.

\subsection{Using the asymptotic profile}
First inspired by our knowledge on long-time asymptotics for the fragmentation equation~\cite{EscoMischler3} and for the growth-fragmentation equation~\cite{MMP2} - note that this field of research has been continuously active for several decades, with remarkable new results inspired by stochastic methods~\cite{canizo2021spectral} or optimal transport~\cite{fournier2021nonexpanding} - and by previous studies of the inverse problem based on the asymptotic steady size distribution~\cite{PZ,DPZ,DHRR}, we first studied the information content of the self-similar asymptotic profile $g$ solution to
\begin{equation}\label{steady:frag}2g(z )+  z g' (z) + \alpha\gamma z^\gamma g(z)=\alpha \gamma \int\limits_z ^\infty \kappa(\f{z}{u}) u^{\gamma-1} g(u ) \diff u,\qquad \int g(z)\diff z=1.
\end{equation}
This idea was guided by the fact that, as proven in Theorem~3.2 of~\cite{EscoMischler3} and Theorem~3.2. of~\cite{MMP2},  under suitable assumptions on $\kappa$ and for $\gamma>0,$ we have
\begin{equation}\label{lim:frag}
\lim _{ t\to \infty }\int \limits _0^\infty\left|u(t, y)-t^{\frac {2} {\gamma }}g\left( t^{\frac {1} {\gamma }} y \right)\right| y\diff y=0,
\end{equation}
so that we can guess that, after rescaling, the profile $g$ may be experimentally observed. The inverse problem can then be reduced to: From observations on $g,$ is it possible to estimate $(\alpha,\gamma,\kappa)$? 

In~\cite{doumic:hal-01501811}, we used the Mellin transform, defined for measure-valued functions by
 \begin{equation}\label{def:Mellin}
 {M}[g](s) \coloneqq \dst\int_0^{+\infty} x^{s-1} g(x)\diff x,   \qquad M[\kappa](s)\coloneqq \dst \int_0^1 x^{s-1} \kappa(\diff x),
 \end{equation}
  and the multiplicative convolution (which plays a similar role for the Mellin transform as the standard convolution product for the Fourier transform)
  \begin{equation}\label{def:conv}
(f\ast g) (x) \coloneqq\dst\int_{\R^+} f(y)\; g\left( \dfrac{x}{y}\right)\dfrac{dy}{y}, \qquad M[f\ast g](s)= M[f](s).M[g](s).
 \end{equation}
We moreover notice that ${M} [z\mapsto z^\gamma g(z)](s)= { M}[g](s+\gamma),$ hence taking the Mellin transform of~\eqref{steady:frag} leads to an explicit formula for $\kappa,$ namely
\begin{equation}\label{eq:Melling}
M[\kappa](s) =1+\f{(2-s)M[g](s)}{\alpha\gamma M[g](s+\gamma)}.
\end{equation}
To justify  rigorously~\eqref{eq:Melling}, we need to prove that $M[g](s+\gamma)$ never vanishes on a vertical strip; we then also need to define its inverse Mellin transform. To do so, we carried out  a detailed analytical study in the complex plane. We also proved a uniqueness result which states that for a wide class of kernels $\kappa,$ for a given $g$, there exists at most one triplet $(\alpha,\gamma,\kappa)$ such that $g$ is solution to~\eqref{steady:frag} (Theorem~1 and~2 in~\cite{doumic:hal-01501811}). These results however revealed of little practical use, for two main reasons.
\begin{itemize}
\item Though $(\alpha,\gamma)$ are uniquely determined, they are obtained by using the asymptotic behaviour of $g(z)$ for $z\to\infty.$ Since we have only access to size samples, we have information on $g$ only on a compact set of $(0,+\infty).$
\item The formula involves Mellin and inverse Mellin transform, so that the inverse problem, as for the second-order asymptotic inverse problem seen above (Theorem~\ref{th:tikhonov}), is severely ill-posed. When numerical implemented, the results failed to be satisfactory.
\end{itemize}
\subsection{Using the long-term dynamics}
We thus turned to the time-evolution dynamics given by~\eqref{lim:frag}. From a sample $(x_1(t_i),\cdots,x_{n_i}(t_i))$ at time $t_i$, we estimate the average polymer size $M_1(t_i)=\f{\int_0^\infty x u(t,x)\diff x}{\int_0^\infty u(t,x)\diff x}$ by the empirical first moment, i.e.
\begin{equation}\label{def:mui}\widehat M_1 (t_i)=\sum\limits_{k=1}^{n_i} x_k (t_i)
\end{equation} and we can relate it to $u(t,x)$  to estimate $(\alpha,\gamma)$. The limit given by~\eqref{lim:frag} implies
\[M_1 (t)\approx_{t\gg 1}  \f{\dst\int t^{\f{2}{\gamma}} y g(t^{\f{1}{\gamma}}y)\diff y}{\dst\int t^{\f{2}{\gamma}}  g(t^{\f{1}{\gamma}})\diff y}=C t^{-\f{1}{\gamma}},
\]
Similarly, for the estimation of $\alpha,$ we notice that integrating~\eqref{steady:frag} and using the mass conservation property $\int \kappa(\diff z)=1$ provides the equality
\begin{equation}\label{eq:alpha}
1=\alpha\gamma \int z^\gamma g(z)\diff z.
\end{equation}
We thus use the $\gamma-$th moment $M_\gamma (t)=\f{\int_0^\infty x^\gamma u(t,x)\diff x}{\int_0^\infty u(t,x)\diff x}$, approximated  by an empirical $\widehat{M_{\hat\gamma}} (t_i)$ defined by
\begin{equation}\label{def:mugammai}
\widehat M_{\hat\gamma} (t_i)=\sum\limits_{k=1}^{n_i} \left(x_k(t_i)\right)^{\hat\gamma} 
\end{equation}
 and the equality
\begin{equation}\label{eq:mugamma}
M_\gamma (t) \approx_{t\gg 1} \f{\int t^{\f{2}{\gamma}} y^\gamma g(t^{\f{1}{\gamma}}y)\diff y}{\int t^{\f{2}{\gamma}} g(t^{\f{1}{\gamma}})\diff y} = \f{\int z^\gamma g(z)\diff z}{t} \implies \alpha \approx_{t\gg 1} \f{1}{\gamma t} \f{1}{M_\gamma (t)}
\end{equation}
so that in~\cite{BTMPSTDX19} we designed the following estimation protocol for $(\alpha,\gamma)$:
\begin{enumerate}
\item \label{step1} On observed data $(\log(t_i),\log(\widehat M_1 (t_i))_{1\leq i\leq n}$ defined by~\eqref{def:mui}, fit  the three parameters $(C,t_{asymp},\hat\gamma)$ of a curve 
\[\left(\log (t), C  -\f{1}{\hat \gamma} \log (\f{t}{t_{asymp}})\1_{t\geq t_{asymp}}\right),
\]
and choose $\widehat \gamma$ as an estimator for $\gamma$. Notice here that we only use a small part of the available data, namely the average size and not the whole size distribution. For each of the four protein fibrils analysed ($\alpha-$Synuclein, associated with Parkinson, bovine $\beta-$lactoglobulin, chicken egg lysozyme and $\beta_2$-microglobulin - this last amyloids being involved in systemic dialysis-related amyloidosis), we had from $7$ to $14$ time points $t_i.$ The fits appeared very satisfactory (see~\cite{BTMPSTDX19}, Fig.~5).
\item  From $\hat\gamma$ and $t_{asymp}$ computed in Step~\ref{step1}, use $\widehat{M_{\hat\gamma} }(t_i)$ defined by~\eqref{def:mugammai} for $t_i\geq t_{asymp}$ and define an estimator of $\alpha$ from~\eqref{eq:mugamma} by
\[ \hat \alpha \coloneqq \f{1}{\hat \gamma t_i} \f{1}{\widehat M_{\hat\gamma} (t_i)}. \]
\end{enumerate}
To validate a posteriori the estimation obtained, we then use the full dataset as follows.
\begin{enumerate}
\item At initial time $t_1=0,$ use the size sample $(x_1(0),\cdots,x_{n_0}(0))$ to obtain an estimate $\widehat f(0,x)$ of the initial size distribution $f(0,x)=\f{u(0,x)}{\int u(0,x)\diff x}$ by a kernel density estimation method~\cite{tsybakov2009nonparametric},
\item Simulate~\eqref{eq:frag_intro} with $u(0,x)=\widehat f(0,x)$, parameters $\hat \alpha,\,\hat \gamma$ and a given kernel $\kappa$,
\item at times $t_i$ of experimental observations, compare the simulated $f(t_i,x)=\f{u(t_i,x)}{\int u(t_i,x)\diff x}$ with the density $\widehat f(t_i,x)$ estimated from the size sample $(x_1(t_i),\cdots,x_{n_i}(t_i))$ by a kernel density estimation method.
\end{enumerate}
We obtained the following quantitative results.
\begin{itemize}
\item We observed a remarkable agreement between the data observed and simulated (see Fig.7~ of~\cite{BTMPSTDX19}). This is all the more satisfactory that only two parameters $\alpha$ and $\gamma$ are necessary to fit the whole time dynamics.
\item The choice of $\kappa$ appears to have only little influence on the goodness of fit, even if we were able to draw some conclusion on whether the breakage occurs more at the edges or at the centre of the fibrils (Fig.8 of~\cite{BTMPSTDX19}).
\end{itemize}

\subsection{Using the short-time dynamics}
Intrigued by the influence of $\kappa$ on the time dynamics, we carried out a thorough numerical investigation in~\cite{tournus2021insights}. Specifically, we developed a statistical test to quantify how  different fragmentation kernels influence the size distribution of particles over time (Fig.4 of~\cite{tournus2021insights}), when all dynamics depart from the same initial distribution. We noticed first that the sharper the initial condition, the greater the influence of $\kappa$; and second, that this influence occurs mainly during an early time-window - even if not too early either, since the initial condition is taken identical in all simulations. Following this study, we turned to a short-time asymptotic development, carried out from a Dirac delta function as an initial distribution. A formal computation shows:
\begin{equation}\label{eq:approx1}
{u }(t+\Delta t,x)\approx u(t,x) -\alpha  \Delta t x^{\gamma}u(t,x)+ \alpha \Delta t \dst\int_{x}^{\infty}\kappa\left(\dfrac{x}{y} \right)y^{\gamma-1} u(t,y)\diff y+ o(\Delta t).
\end{equation}
Departing at time $t=0$ from  a Dirac delta function $\delta(x-1)$ at $x= 1$, we have
\begin{equation*}
{u }(\Delta t,x)\approx \delta(x-1)(1 -\alpha  \Delta t) + \alpha \kappa(x) \Delta t  +o(\Delta t), 
\end{equation*}
and thus we obtain an estimation formula for the kernel $\kappa$: if we have ${\widehat u(\Delta t,\cdot) }$ an estimate, obtained from observations on the size distribution ${u(\Delta t,\cdot) },$ we could define an estimate for $\kappa$ as
\begin{equation}\label{eq:estim1}
 \widehat\kappa(x) \coloneqq \dfrac{1}{\alpha \Delta t}\left({u }(\Delta t,x) - (1-\alpha \Delta t) \delta(x-1)\right) +o(1), \qquad  \Delta t \ll 1.
\end{equation}

How to evaluate the error made between $\kappa$ and $\widehat\kappa$? We first recall the definitions of the Total Variation norm and the Bounded Lipshitz norm, the latter being more appropriate  to evaluate the discrepancy between two measures. 

\begin{definition}\label{def:norms}
\label{def:measure_valued_sol}
Let $\M(\R^+)$ be the set of Radon measures whose support belongs to $\R^+$. 
The Total Variation (TV)  norm of
the signed measure  $\mu \in \M(\R^+)$ is
defined as
\begin{equation}
\label{TVdef}
    \|\mu\|_{TV} \coloneqq \sup \{\dst\int_{R^+}  \varphi(x)d\mu(x), \quad \varphi \in \mathcal{C}(\R^+)\cap L^1(d|\mu|), \; \|\varphi\|_{\infty} \leq 1\},
\end{equation}
whereas its Bounded Lipshitz (BL) norm is defined by 
\begin{equation}\label{def:BL}
    \|\mu\|_{BL} \coloneqq \sup \{\dst\int_{R^+}  \varphi(x)d\mu(x), \quad \varphi \in \mathcal{C}(\R^+)\cap L^1(d|\mu|), \; \|\varphi\|_{\infty} \leq 1, \; 
    \|\varphi '\|_{\infty}\leq 1\}.
\end{equation}
\end{definition}

A second preliminary step is to define measure-valued solutions to~\eqref{eq:approx1}, since we want to depart from a measure-valued initial condition. This is given by the following definition.

\begin{definition}[Weak solution for \eqref{eq:frag_intro} - Definition~2.1. of~\cite{doumic2021inverse}]
 A family $(\mu_t)_{t\geq 0} \subset \M(\R^+)$, where $\M(\R^+)$ denotes the space of Radon measures on $\R_+,$ is called a measure-valued solution to~\eqref{eq:frag_intro}, with initial data
 $\mu_0 \in \M(\R^+)$ satisfying $\supp (\mu_0)\subset [0,L]$, if the mapping $t\to \mu_t$ is narrowly continuous and for all $\varphi \in \mathcal{C}(\R^+)$ { such that $x\mapsto \varphi(x)/(1+x)$ is bounded on $[0,\infty)$,} and all $t \geq 0$,
{\begin{equation}
\label{ppcm1}
 \int_{\R^+} \varphi(x) d\mu_t(x)=  \int_{\R^+} \varphi(x) d\mu_0(x) + \dst\int_0^t ds \dst\int_{\R^+}
   d\mu_s(x)  \alpha x^{\gamma}\left( - \varphi(x)
 + \dst\int_{0}^{1} d\kappa(z)\varphi(xz)\right).
\end{equation}
}
\end{definition}

The class of functions used to define the weak solutions ensures the finiteness of $\int (1+x)\diff \mu_t (x).$ 
Despite the numerous studies carried out on fragmentation equations, the question of existence and uniqueness of measure-valued solutions to~\eqref{eq:frag_intro} for a large class of kernels $\kappa$  appeared yet unsolved. We thus formulated the following result, which moreover provides an explicit decomposition of such solutions.
\begin{theorem}[From Theorems~2.2, 2.4 and~2.6 of~\cite{doumic2021inverse}]
Let $\kappa$ satisfy~\eqref{as:kappa}, $\alpha>0$, $\gamma\geq 0$ and $\mu_0 \in  \mathcal{M^+}((0,L))$ for some $L>0$. There exists a unique solution $\mu\in \mathcal{C}(\R^+,\mathcal{M^+}(\R_+))$ to~ \eqref{eq:frag} in the sense of Definition  \ref{def:measure_valued_sol}. Moreover, this unique solution 
\begin{itemize}
\item preserves the mass:  $\dst\int x \mu_t(\diff x)=\int x \mu_0(\diff x)$ for any $t\geq 0$, 
\item is nonnegative,
\item satisfies $\supp(u(t,\cdot))\subset (0,L)$ for any $t\geq 0,$ 
\item  is explicitely defined as the following series, which is absolutely convergent in the TV norm for any $t\geq 0$:
\begin{equation}
\label{representation_solution}
  \mu_t = e^{-\alpha x^{\gamma}t} \mu_0
  +  \sum_{n=0}^{\infty} (\alpha t)^n
  \dst\int_0^{\infty} \ell^{n\gamma}a_n\left(\dfrac{x}{\ell}\right)\mu_0(\ell) \dfrac{d\ell}\ell{},
 \end{equation}
 where the sequence $a_n$ is defined by induction, for $x\in [0,1]$, by
 \begin{equation}
 \label{induction}
  a_0(x)=0, \quad a_{n+1}(x) = \dfrac{1}{n+1}\left(-x^{\gamma}a_n(x)+ 2\dst\int_x^{\infty} y^{\gamma-1} \kappa \left(\dfrac{x}{y} \right)a_n(y)\diff y
  + \kappa(x) \dfrac{(-1)^n}{n!}\right).
 \end{equation}
\end{itemize}
\end{theorem}
\begin{proof}
Let us sketch the main steps of the proof. 

\

{\bf Step 1:} uniqueness and, if the solution is proved to be nonnegative,  $\supp(u(t,\cdot))\subset (0,L)$ (Theorem~2.2 in~\cite{doumic2021inverse}). By choosing appropriate weight functions, we  prove that a nonnegative solution satisfies $\supp(u(t,\cdot))\subset (0,L)$ for any $t\geq 0,$ and then, with a Gronwall lemma, we prove the estimate
\[\|\mu_t\|_{TV} \leq  \|\mu_0\|_{TV} e^{\alpha (2{L})^{\gamma} 3 t}, \]
which implies uniqueness.

\

{\bf Step 2:} the representation formula~\eqref{representation_solution} (Theorem~2.4 in~\cite{doumic2021inverse}). We first obtain the representation formula for  the fundamental solution $\mu_t^F,$ 
 defined as a solution departing from $\mu_0(x)=\delta(x-1),$ namely:
 \begin{equation}\label{eq:fundamental}\mu_t^F(x)=e^{-\alpha t}\delta(x-1)+\sum\limits_{n=1}^\infty (\alpha t)^n a_n(x)
.\end{equation}
By a scaling property, we deduce the solution for $\mu_0(x)=\delta(x-\ell),$ and we conclude by the superposition principle. 

\

{\bf Step 3:} nonnegativity of the representation formula~\eqref{representation_solution} (Theorem~2.6 in~\cite{doumic2021inverse}). This last step is more delicate than it can seem at first sight, and we have not found nonnegativity results for cases where both the solution and the fragmentation kernel are measure-valued functions  in the literature. We prove this result by an approximation strategy on the fundamental solution $\mu_t^F$. Regularising both $\kappa$ and the initial condition $\delta(x-1),$  we can apply a nonnegativity result~\cite{Melzak}; we then conclude by weak convergence of the approximate sequence.
\end{proof}

Equipped with this well-defined notion of measure-valued solutions, we then turn to our inverse problem, and prove successive error estimates:
\begin{enumerate}
\item \label{IP1} error estimate, in TV norm, for perfectly observed data (no noise), with perfectly monodisperse initial condition ($\mu_0=\delta(x-1)$),
\item \label{IP2} error estimate, in BL norm, for noisy data and nearly monodisperse initial condition ($\Vert \mu_0-\delta(x-1)\Vert_{BL}$ small),
\item \label{IP3} error estimate, via the Mellin transform, for perfectly observed data (no noise) and disperse initial condition (arbitrary $\mu_0$).
\end{enumerate}

\

\subsubsection*{Error estimate~\eqref{IP1}: monodisperse initial condition, no noise}
From~\eqref{eq:fundamental}, we are led to precise the formal approximation~\eqref{eq:estim1}, and define
\begin{equation}\label{eq:estim2}
\kappa^{est}(t)\coloneqq \f{\mu_t^F -e^{-\alpha t}\delta(x-1)}{\alpha t}.
\end{equation}
From~\eqref{representation_solution} we immediately obtain (Theorem~3.1 in~\cite{doumic2021inverse}):
\begin{equation}\label{error:estim1}
 \Big\|\kappa^{est}-\kappa\Big\|_{TV} \leq t\left(\alpha \sum\limits_{n=0}^{\infty} (\alpha T)^n \|a_{n+2}\|_{TV}\right),\,\,\forall t\in (0, T].
\end{equation}

\

\subsubsection*{Error estimate~\eqref{IP2}: nearly monodisperse initial condition, noisy data}
From~\eqref{error:estim1} it seems that the smaller the observation time $t,$ the more precise the estimate $\kappa^{est}.$ This fails in practice, where the noise comes from at least two sources: 
\begin{itemize}
\item the initial condition cannot be as precise as a Dirac delta function. Even if we could imagine an experimental setup selecting polymers of the same size, there will always remain some heterogeneity. 
\item Typical experimental measurements being polymer size samples,  $\mu_t$ as well as $\mu_0$ can only be observed up to a noise.
\end{itemize}
For these two sources of noise, the distance in BL-norm is well-adapted. We thus formulate the following error estimate result.

\

\begin{theorem}[Theorem~3.5 in~\cite{doumic2021inverse}]
\label{thm:stab}
Let $\kappa$ satisfy \eqref{as:kappa} and $\supp(\mu_0^{\ep_0})\subset [0,L]$ such  that
\begin{equation*}
 \|\mu_0^{\ep_0} - \delta(x-1)\|_{BL} \leq \ep_0.
\end{equation*}
Let $\mu_t^{\ep_0}$ the unique solution to the fragmentation equation 
\eqref{eq:frag} with initial condition $\mu_0^{\ep_0}$.
Let $\mu_0^{\ep_0,\ep_1}$ and 
 $\mu_t^{\ep_0,\ep_2}$ be noisy observations of the respective measures $\mu_0^{\ep_0}$ and $\mu_t^{\ep_0}$
 such that 
 \begin{equation*}
 \|\mu_0^{\ep_0,\ep_1}- \mu_0^{\ep_0}\|_{BL} \leq \ep_1,\qquad 
   \|\mu_t^{\ep_0,\ep_2}- \mu_t^{\ep_0}\|_{BL} \leq \ep_2.
 \end{equation*}
 {Assume moreover either $\gamma\geq 1$ or $\supp (\mu_0)\subset [m,{L}]$ with $m>0$}.
Then, for all $0\leq t\leq T$, there exist constants ${C_1(T,\alpha)}$ and ${C_2(L,T,\alpha,{\gamma,m^{\gamma-1}})}$ such that 
\begin{equation}
\label{thm:stabE4}
 \Big\|\dfrac{\mu_t^{\ep_0,\ep_2} - e^{-\alpha t} \mu_0^{\ep_0,\ep_1} }{\alpha t}-\kappa\Big\|_{BL} 
 \leq {C}_1  t + \dfrac{{C}_2 \ep_0+\ep_1+\ep_2}{\alpha t}. 
\end{equation}
\end{theorem}

\

The proof is immediate once a stability result for the solution to the fragmentation equation in BL-norm is obtained (Theorem~2.11 in~\cite{doumic2021inverse}): this stability is obtained only for $\gamma\geq 1$ or $\supp (\mu_0)\subset [m,{L}]$ with $m>0$, hence these extra assumptions. The form of the inequality~\eqref{thm:stabE4} is very interesting, since it has the exact form of the usual bias-variance decomposition for regularisation of inverse problems as well as nonparametric density estimation: a balance between two terms, one vanishing with the regularisation parameter, the other exploding with it but multiplied by the noise level, so that the observation time $t$ identifies with a regularisation parameter, and is optimal if $t=O(\sqrt{\ep_0+\ep_1+\ep_2})$. Intuitively, this corresponds to a time large enough so that "enough" (compared to the noise level) polymers have divided once, but small enough so that not too many have divided twice (term $C_1 t,$ which shows that the distance increases with time).

\

\subsubsection*{Error estimate~\eqref{IP3}: arbitrary initial condition, no noise}
In many practical cases, it seems very difficult, if not impossible, to have monodisperse or almost monodisperse initial condition, so that $\mu_0$ is far away from a Dirac delta function. The same strategy can however be followed.

For simplicity, let $\kappa=\kappa(x)\diff x$ and $\mu_0=u_0(x)\diff x$ admit densities with respect to the Lebesgue measure, denote $\mu_t(\diff x)=u(t,x)\diff x$ and define
\begin{equation}\label{def:kappaestmu0}
F^{est}(u _0,\kappa; t, x)=\dfrac{u (t,x) - e^{-\alpha tx^{\gamma}}u _0(x)}{\alpha t}.
\end{equation}
 Thanks to the superposition principle, we have the following inequality (Corollary~3 in~\cite{doumic2021inverse}), for a given constant $C$,
 
 \
 
\begin{equation}
\label{cor:generic1}
\Big\|{F^{est}(u _0,\kappa; t)}-w_0\ast\kappa(x)\Big\|_{TV} \leq  C{L}^{2\gamma }\|{u_0}\|_{ TV}\, t,\,\,\forall t\in (0, T],
\end{equation}

\

where $\ast$ denotes the multiplicative convolution defined by~\eqref{def:conv}
 and $w_0(x)=x^\gamma u_0(x).$ To have an estimate for $\kappa,$ we thus need to invert the multiplicative convolution. 
 Recalling the formula for the Mellin transforms~\eqref{def:Mellin},   a candidate to estimate $M[ \kappa]$ is thus
  \begin{equation}\label{def:Mellinkappa}M[\kappa]^{est} (s;t)\coloneqq \f{ {M}[u(t,\cdot)] (s)-  {M}\left[x\mapsto e^{-\alpha tx^{\gamma}}u _0(x)\right](s)}{\alpha t { M} [u_0](s+\gamma)}.
  \end{equation}
  
 \
 
  The common point between~\eqref{eq:Melling} and~\eqref{def:Mellinkappa} is the presence in the denominator of  the Mellin transform, taken in $s+\gamma,$ of a size distribution observed. For large $\gamma,$ this may reveal very noisy, since this gives a lot of  weight to the largest particles. However, contrarily to the case of formula~\eqref{eq:Melling}, we managed to prove an error estimate between $M[\kappa]^{est} (s;t)$ and $M[\kappa](s)$, thanks to the Mellin transform of the series representation~\eqref{representation_solution} and under regularity assumptions on $\kappa$ and $u_0$. We refer the reader to Theorems~4.4 and 4.6 in~\cite{doumic2021inverse} for a detailed presentation of the results, which contain an estimate of the distance between $\kappa$ and $\kappa^{est}\coloneqq M^{-1}[M[\kappa]^{est}(\cdot;t)]$, in a certain weighted norm, in the order of $ C^{st}\Delta t.$ This order of magnitude is coherent with the estimate~\eqref{error:estim1}; however the constant $C^{st}$ is not explicit, and the observation noise is not yet taken into account. This question of stability with respect to noise, together with a numerical investigation, are perspectives for future work.
\bibliographystyle{plain}

\bibliography{Tout_Marie}

\end{document}